\documentclass[12pt]{amsart}

\usepackage{amsmath,amssymb,amsthm,enumitem}

\usepackage{multirow}
\usepackage{tikz}
\usepackage{amsfonts}

\usepackage{tikz}
\usetikzlibrary{matrix,arrows,decorations.pathmorphing}

\usepackage{enumitem}

\makeatletter
\@namedef{subjclassname@2010}{%
  \textup{2010} Mathematics Subject Classification}
\makeatother

\DeclareMathOperator{\Gal}{Gal}

\def\Tr{\operatorname{Tr}}

\renewcommand{\phi}{\varphi}

\newtheorem{theorem}{Theorem}[section]
\newtheorem*{thm}{Theorem}

\newtheorem{proposition}[theorem]{Proposition}
\newtheorem{lemma}[theorem]{Lemma}
\newtheorem{corollary}[theorem]{Corollary}

\theoremstyle{definition}
\newtheorem{remark}[theorem]{Remark}

\newtheorem{example}[theorem]{Example}

\def\cqfd{
{\hfill
\kern 6pt\penalty 500
\raise -1pt\hbox{\vrule\vbox to 5pt{\hrule width 4pt
\vfill\hrule}\vrule}}
\break}

\font\tengoth=eufm10
\font\sevengoth=eufm7
\font\fivegoth=eufm5
\newfam\gothfam
\textfont\gothfam=\tengoth\scriptfont\gothfam=\sevengoth\scriptscriptfont\gothfam=\fivegoth
\def\goth{\fam\gothfam\tengoth}

\def\Sgoth{{\goth S }}
\def\Agoth{{\goth A }}

\title{Differential uniformity of polynomials of degree 10}

\author[Aubry]{Yves Aubry} 
\address[Aubry]{Institut de Math\'ematiques de Toulon - IMATH, Universit\'e de Toulon, France}
\email{yves.aubry@univ-tln.fr}

\address[Aubry]{Institut de Math\'ematiques de Marseille - I2M, Aix Marseille Univ, UMR 7373 CNRS, France}
\email{yves.aubry@univ-amu.fr}

\begin{document} 


\baselineskip=17pt


\begin{abstract}
We prove that polynomials of degree 10 over finite fields of even characteristic with some conditions on theirs coefficients have 
a differential uniformity greater than or equal to 6 over ${\mathbb F}_{2^n}$ for all $n$ sufficiently large.

\end{abstract}




\thanks{This work is partially supported by the French Agence Nationale de la Recherche through the SWAP project under Contract ANR-21-CE39-0012.}

\maketitle

\section{Introduction}

Differential uniformity of polynomials over finite fields is a measure of non-linearity and resistance against differential attacks in cryptography. Formally, the differential uniformity $\delta_{{\mathbb F}_{q}}(f)$ of a polynomial  $f\in {\mathbb F}_{q}[x]$ 
over the finite field ${\mathbb F}_{q}$ with $q$ elements
is defined as
the maximum number of solutions 
of the set of equations $f(x+ \alpha) - f(x) = \beta$
where $\alpha$
and $\beta$ belong to ${\mathbb F}_q$ with $\alpha$ non-zero (see \cite{Nyberg} where it has been first introduced).
For practical cryptographic applications, a particular study has been made over finite fields of characteristic 2, which will be the 
framework of our work here.
Polynomials over ${\mathbb F}_{2^n}$ with low differential uniformity are 
 highly sought after, especially those with the smallest possible one, namely equal to 2. The functions associated with these polynomials are called APN (Almost Perfect Nonlinear) functions, and exhaustive research suggests that they are very rare. In fact, Voloch proved in \cite{Felipe} that almost all polynomials have a differential uniformity essentially equal to their degree.
 Even better,  Aubry,  Herbaut and  Voloch in \cite{YvesFabienFelipe}  showed that, for a set
 of specific odd degrees, not almost all but indeed all polynomials of these degrees have maximal differential uniformity for $n$ sufficiently large.
 Moreover, these results have been extended in \cite{AHI-J.Algebra} to infinitely many explicit even degrees and in 
\cite{AHI-JNT} to some trinomials of degree divisible by 4.

The study of the differential uniformity of low-degree polynomials was conducted by Voloch in \cite{Felipe}. Apart from the trivial case of polynomials of degrees less than 4, he addressed the cases of degrees 5, 6, and 7 (the case of degree 8 is reduced to that of lower degrees), and he stopped at degree 9.

The main result of our paper concerns polynomials of degree 10 over finite fields of even characteristic. The methods developed in 
\cite{AHI-J.Algebra}  and 
\cite{AHI-JNT}, although applicable to even-degree polynomials, cannot be applied mutatis mutandis to our situation. Therefore, we are led to develop here a specific approach that does not rely on the 
description of the locus of polynomials with non-distinct critical values,
as was the case in 
\cite{YvesFabienFelipe},
\cite{AHI-J.Algebra}  and \cite{AHI-JNT}.

Precisely, we prove the following results.

\begin{thm}(Theorem \ref{Main} and Theorem \ref{cas_a_1=a_3=0})
Let  $f=\sum_{i=0}^{10}a_{10-i}x^i\in{\mathbb F}_{2^n}[x]$
be a polynomial
of degree $10$.

\noindent
1) If
\begin{enumerate}[label=(\roman*)]

\item $a_1a_3\not=0$ and,

\item ${\Tr}_{{\mathbb F}_{2^n}/{\mathbb F}_2}  \left( \frac{a_1a_4+a_5}{a_1^2a_3} \right)=0$ and,

\item $a_1^2a_4^2+a_5^2+a_1^7a_3+a_1^4a_3^2+a_1^2a_3a_5+a_3a_7\not=0$,
\end{enumerate}
then $\delta_{{\mathbb F}_{2^n}}(f)\geq 6$ if $n$ is sufficiently large (namely if $n\geq 13$).

\bigskip
\noindent
2) Suppose that
 $a_1=a_3=0$,
and suppose that there exists $\alpha\in{\mathbb F}_{2^n}^{\ast}$ such that:
\begin{enumerate}[label=(\roman*)]
\item
$c:= \frac{\alpha^2 a_5+ a_7}{\alpha}\not=0$
 and the polynomial $R_3(x):=x^3+bx^2+c^2$ 
has all its  roots in ${\mathbb F}_{2^n}$ where
$b:=\frac{\alpha^5+\alpha a_4+ a_5}{\alpha}$,

\item
and  ${\Tr}_{{\mathbb F}_{2^n}/{\mathbb F}_2}  \left( \frac{\alpha^5+\alpha a_4 + a_5}{\alpha^3}\right)=0$,
\end{enumerate}
then $\delta_{{\mathbb F}_{2^n}}(f)=8$ if $n$ is sufficiently large (namely if $n\geq 15$).
\end{thm}

\begin{remark}
 Functions which are APN over infinitely many extensions of 
 the base field
are called {\sl exceptional} APN. Aubry, McGuire and Rodier conjectured in \cite{A-McG-R} that, up to a certain equivalence, the 
Gold functions $f(x)=x^{2^k+1}$ and the Kasami-Welch functions $f(x)=x^{2^{2k}-2^k+1}$ are the only exceptional APN functions.
The results of the present paper imply that the polynomials of degree 10 satisfying the conditions of our theorem are {\sl a fortiori} not exceptional APN: we recover a known result
since the conjecture in the case of polynomials $f$ of degree $2e$ with $e$ odd and when $f$ contains a term of odd degree has been proved by Aubry, McGuire and Rodier  in \cite{A-McG-R}.
\end{remark}





Section \ref{monodromy} is dedicated to the strategy of introducing a polynomial whose splitting field produces a Galois extension in which we will prove the existence of a place which totally splits using Chebotarev's density theorem.
Section \ref{main} focuses on the first part of the previous theorem and relies on Morse polynomial theory to obtain monodromy groups equal to the symmetric group. Finally  Section \ref{Nul} concentrates on the second part of the previous theorem and uses the characterization of the Galois groups of quartic polynomials through their quadratic and cubic resolvents.

\section{Monodromy groups, Morse polynomials and geometric extensions}\label{monodromy}

Let $f(x) =\sum_{i=0}^{10}a_{10-i}x^i \in {\mathbb F}_{q}[x]$, where $q=2^n$, be a polynomial of degree $m=10$ (so $a_0$ is always supposed to be non-zero). 
Let $\alpha\in{\mathbb F}_q^{\ast}$ and consider  $D_{\alpha}f(x)=f(x+\alpha)+f(x)$  the derivative of $f$ with respect to $\alpha$.
 By definition, the differential uniformity of $f$ is given by
 $$\delta(f):=\max_{(\alpha,\beta)\in{\mathbb F}_q^{\ast}\times{\mathbb F}_q}\sharp\{x\in{\mathbb F}_{q} \mid D_{\alpha}f(x)=\beta\}.$$

Consider  
the unique polynomial $L_{\alpha}f$ 
such that
 $L_{\alpha}f \left( x(x+\alpha)\right)=D_{\alpha}f(x)$
 (see Proposition 2.3 of  \cite{YvesFabienFelipe} for the existence and the unicity of such a polynomial $L_{\alpha}f$)
 and let us denote by $d$ its degree.
 A simple calculation gives :
\begin{multline}\label{Dalphaf}
D_{\alpha}f(x)=(a_0\alpha^2 +a_1\alpha) x^8 
+ a_3\alpha x^6 
+ a_3\alpha^2 x^5
+ (a_3\alpha^3+a_4\alpha^2+a_5\alpha) x^4\\
+ a_3\alpha^4 x^3
+ (a_0\alpha^8+a_3\alpha^5+a_4\alpha^4+a_7\alpha) x^2
+ (a_1\alpha^8+a_3\alpha^6+a_5\alpha^4+a_7\alpha^2) x\\
+a_0\alpha^{10}+a_1\alpha^9+a_2\alpha^8+a_3\alpha^7+a_4\alpha^6+a_5\alpha^5+a_6\alpha^4+a_7\alpha^3+a_8\alpha^2+a_9\alpha
\end{multline}
and
\begin{multline}\label{Lalphaf}
 L_{\alpha}f(x)=(\alpha^2a_0+\alpha a_1) x^4 + \alpha a_3 x^3 +(\alpha^6a_0+\alpha^5a_1+\alpha^2a_4+\alpha a_5)x^2\\
 + 
 (\alpha^7a_1+\alpha^5 a_3+\alpha^3 a_5 + \alpha a_7)x \\
 +
 \alpha^{10}a_0+\alpha^9a_1+\alpha^8a_2+\alpha^7a_3+\alpha^6a_4+\alpha^5a_5+\alpha^4a_6+\alpha^3a_7+\alpha^2a_8
 +\alpha a_9.
\end{multline}

Then we consider  the splitting field  $F$ of the
polynomial 
$L_{\alpha}f(x)-t$
 over the field ${\mathbb F}_q(t)$
 with $t$ a transcendental element over  ${\mathbb F}_q$
and we set  ${\mathbb F}_q^{F}$ to be the algebraic closure of  $\mathbb F_q$ in $F$.
We consider now the Galois groups $G=\Gal(F/{\mathbb F}_q(t))$ and $\overline G=\Gal(F/{\mathbb F}_q^{F}(t))$ which are respectively  the arithmetic and geometric monodromy groups of $L_{\alpha}f$.

If $u_0,\ldots,u_{d-1}$ are the roots of $L_{\alpha}f(x)=t$, then 
we will denote by $x_i$
a root of
$x^2+ \alpha x=u_i$.
So the $2d$ elements $x_0,x_0+\alpha,\ldots,x_{d-1},x_{d-1}+ \alpha$
are the solutions of
$D_{\alpha}f(x)=t$.
Then we consider $\Omega={\mathbb F}_q(x_0,\ldots, x_{d-1})$ 
 the compositum of the fields $F(x_i)$ and  ${\mathbb F}_q^{\Omega}$ 
the algebraic closure of ${\mathbb F}_q$ in $\Omega$.
We set  also $\Gamma=\Gal(\Omega/F)$ and $\overline\Gamma=\Gal(\Omega/F{\mathbb F}_q^{\Omega})$.
Then we have the following diagram where the constant field extensions from $k={\mathbb F}_{2^n}$ are drawn and where ${\mathcal C}_{F}$ and ${\mathcal C}_{\Omega}$ stand for the smooth projective algebraic  curves associated to the function fields $F$ and $\Omega$:

%
%
%
%

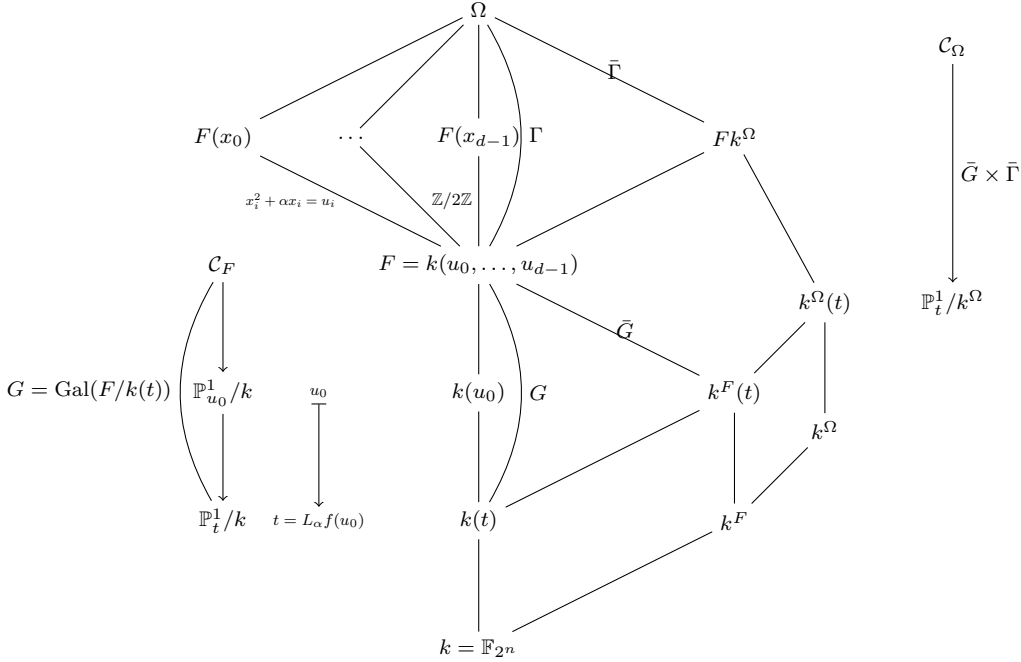
\begin{center}
\tiny
\begin{tikzpicture}[node distance=1.7cm]

 \node (k)                  {$k= \mathbb{F}_{2^n}$};
 \node (kt)   [above of=k]          {$k(t)  $};
 \node (vide) [above of=kt] {$k(u_0)$};
 \node (F)  [above of=vide]              {$F=k(u_0,\ldots,u_{d-1})$};
 \node (Fxi) [above of=F]  {$F(x_{d-1})$};
 \node (Omega)  [above of=Fxi]   {$\Omega$};
 
 \node (Fx1)  [left of=Fxi]   {$\ldots$};
 \node (Fxd1)  [left of=Fx1]  {$F(x_{0})$};
 
  \node (A)  [right of=Fxi]   {};
 \node (B)  [right of=A]   {$Fk^{\Omega}$};

\node (vide2) [right of=vide] {}; 
 \node (komegat)   [right of=vide2]          {$k^{F}(t)$};
 \node (vide7)  [right of=komegat]          {};
 \node (komega)   [below of=komegat]          {$k^{F}$};
 \node (kbarre) [above right of=komega] {$k^{\Omega}$};
 \node (kbarret) [above of=kbarre] {$k^{\Omega}(t)$};
   \node (P1tbis)  [right of=kbarret]          {$\mathbb{P}^1_t  / k^{\Omega}$};
   \node (I) [above of=P1tbis] {} ;
     \node (J) [above of=I] {$\mathcal{C}_{\Omega}$} ;

\node (vide3) [above of=kbarret] {}; 
 \node (vide4) [above of=vide3] {}; 
   \node (vide5) [left of=vide] {};
   \node (vide6) [left of=kt] {};
   \node (P1u) [left of=vide5] {$\mathbb{P}^1_{u_0}/k$};
   \node (P1t) [below of=P1u] {$\mathbb{P}^1_t/k$};
   \node (u) [scale=0.75,right of=P1u] {$u_0$};
 \node (gdeu) [scale=0.75,right of=P1t] {$ t=L_{\alpha}f(u_0)$};

  \node (Cg) [above of=P1u] {$\mathcal{C}_{F}$};
      
 \draw (F)   to node[left, midway,scale=0.65]  {$x_i^2 + \alpha x_i = u_i \ \ $ }  (Fxd1);
 
 \draw (F)  to node[left, midway,scale=0.8]  {$\mathbb{Z} / 2 \mathbb{Z}$} (Fxi) ;
 \draw (F)   -- (Fx1);
 \draw (Fx1)  -- (Omega);
 \draw (Fxd1)  -- (Omega);
 \draw (Fxi)  -- (Omega);
 \draw (k) -- (kt);
 \draw (kt) -- (vide);
  \draw (F) -- (vide);
 \draw (kt) -- (komegat);
 \draw (komega) -- (komegat) ; 
\draw (k) -- (komega); 
 
 \draw (komega) -- (kbarre);   
 \draw (kbarre) -- (kbarret);   
 \draw (komegat) -- (kbarret); 
\draw(P1t) to [bend left]  node [bend left,midway,left] {$G=\Gal(F/k(t))$} (Cg) ;
\draw(F) to [bend left]  node [bend left,midway,right] {$G$} (kt) ;
 \draw (komegat) to node [right, midway, right]  {$\bar G$} (F) ; 

\draw(Omega) to [bend left]  node [bend left,midway,right] {$\Gamma$} (F) ;

\draw[->] (P1u) -- (P1t) ;
\draw[->] (Cg) -- (P1u) ;
\draw[|->] (u) -- (gdeu) ;
\draw[->] (J) -- (P1tbis) node[right, midway]{$\bar G \times \bar\Gamma $} ;

\draw (F) -- (B);
\draw (B) to node [right, midway, right]  {$\bar \Gamma$} (Omega);
\draw (kbarret) -- (B) ;
\end{tikzpicture}

\end{center}

The purpose here is to apply the Chebotarev density theorem in order to get the existence of an element 
$\beta$ in a finite extension $\mathbb F$ of ${\mathbb F}_{2^n}$ such that the polynomial 
$D_{\alpha}f(x)+\beta$ 
splits in $ \mathbb{F}[x]$. Indeed, the Chebotarev  theorem describes the distribution of places in a 
Galois extension of number fields or in a
geometric Galois extension of function fields of one variable over a finite field. It states that for any conjugacy class of the Galois group, there exists a density of places whose Frobenius automorphism falls within that class. For an unramified place, the associated conjugacy class, that is the Artin symbol attached to this place, is reduced to the identity automorphism if and only if the place splits in the Galois extension. 

So the point is to work with 
 a {\sl geometric} (or {\sl regular}) Galois extension $\Omega/{\mathbb F}_q(t)$, that is with no constant field extension.
In other words, we want to find an $\alpha$ such that  $G=\overline{G}$ and $\Gamma=\overline{\Gamma}$.

The regularity of the extension $\Omega/F$ will be derived from Proposition 4.6 of \cite{YvesFabienFelipe} (and a generalization)
and is related to a Trace equation.
The regularity of the extension $F/{\mathbb F}_{q}(t)$, for its part, will come from the theory of Morse polynomials in Section \ref{main}
and from quadratic and cubic resolvents in Section \ref{Nul}.





\section{The result with $a_1a_3\not=0$}\label{main}

Let $f(x) =\sum_{i=0}^{10}a_{10-i}x^i \in {\mathbb F}_{2^n}[x]$ be a polynomial of degree $m=10$  with $a_1\not=0$ and $a_3\not= 0$.
Consider the choice:
$$\alpha=a_1/a_0.$$

Then Formulas (\ref{Dalphaf}) and (\ref{Lalphaf}) give that the polynomial 
$$D_{\frac{a_1}{a_0}}f(x)=\frac{a_1^3a_3}{a_0^3} x^6 +\cdots$$
 has degree 6 and the polynomial
 \begin{align*}
 L_{\frac{a_1}{a_0}}f(x) & =\frac{a_1a_3}{a_0}x^3
 +\left(\frac{a_1^2a_4}{a_0^2}+\frac{a_1a_5}{a_0}\right) x^2
+ \left(\frac{a_1^8}{a_0^7}+\frac{a_1^5a_3}{a_0^5}+\frac{a_1^3a_5}{a_0^3}+\frac{a_1a_7}{a_0}\right)x \\
&
+\frac{a_1^{8}a_2}{a_0^8}+\frac{a_1^7a_3}{a_0^7} + \frac{a_1^6a_4}{a_0^6}+\frac{a_1^5a_5}{a_0^5}+
\frac{a_1^4a_6}{a_0^4}+\frac{a_1^3a_7}{a_0^3}+\frac{a_1^2a_8}{a_0^2}+\frac{a_1a_9}{a_0}\\
\end{align*}
 has degree $d=3$.

Recall that a polynomial $g\in{\mathbb F}_{2^n}[x]$ is said to be Morse (see the Appendix of Geyer to the paper \cite{JardenRazon})
if it has odd degree, if the critical points of $g$ are non degenerate (i.e. the derivative $g'$ and the second Hasse-Schmidt derivative $g^{[2]}$ have no common roots) and if the critical values of $g$ are distinct ($g$ does not take the same value at different zeros of $g'$).
We have:

\begin{proposition}\label{Morse}
Let $f=\sum_{i=0}^{10}a_{10-i}x^i\in{\mathbb F}_{2^n}[x]$ be a polynomial
of degree $10$.
If
\begin{enumerate}[label=(\roman*)]
\item $a_1a_3\not=0$,
and
\item $a_0^4a_1^2a_4^2+a_0^6a_5^2+a_1^7a_3+a_0^2a_1^4a_3^2+a_0^4a_1^2a_3a_5+a_0^6a_3a_7\not=0$,
\end{enumerate}
then the polynomial $L_{\frac{a_1}{a_0}}f$ is Morse.
\end{proposition}

\begin{proof}
Let $f=\sum_{i=0}^{10}a_{10-i}x^i$ be as in the theorem and set $g=L_{\frac{a_1}{a_0}}f$.
The polynomial $g$ has odd degree (its degree is 3) and the critical values of $g$ are obviously distinct since $g'$ has degree 2 and thus has only one double root.

Now let us find a necessary and sufficient condition for the critical points of $g$ to be nondegenerate.
We have $g'(x)=\frac{a_1a_3}{a_0}x^2+ 
\frac{a_1^8}{a_0^7}+\frac{a_1^5a_3}{a_0^5}+\frac{a_1^3a_5}{a_0^3}+\frac{a_1a_7}{a_0}$.

Recall that the Hasse-Schmidt derivative $g^{[2]}$  is defined by 
the equality 
$g(t+u) \equiv g(t)+g'(t)u + g^{[2]}(t)u^2 \pmod{u^3}$
where $u$ and $t$ are independent variables.
Then we get here: $g^{[2]}(x)=\frac{a_1a_3}{a_0} x + \frac{a_1^2a_4}{a_0^2}+\frac{a_1a_5}{a_0}$ which has
$x=\frac{a_0a_5+a_1a_4}{a_0a_3}$ as a root. And this root is also a root of $g'$ if and only if 
$$a_0^4a_1^3a_4^2+a_0^6a_1a_5^2+a_1^8a_3+a_0^2a_1^5a_3^2+a_0^4a_1^3a_3a_5+a_0^6a_1a_3a_7=0.$$
Thus condition (ii) ensures that the polynomial $g=L_{\frac{a_1}{a_0}}f$ is Morse.
\end{proof}







\begin{theorem}\label{Main}
 For $n$ sufficiently large, namely for $n\geq 13$, 
for all polynomials $f=\sum_{i=0}^{10}a_{10-i}x^i\in{\mathbb F}_{2^n}[x]$
of degree $10$ such that :

\begin{enumerate}[label=(\roman*)]

\item $a_1a_3\not=0$ and,

\item ${\Tr}_{{\mathbb F}_{2^n}/{\mathbb F}_2}  \left( \frac{a_1a_4+a_5}{a_1^2a_3} \right)=0$ and,

\item $a_1^2a_4^2+a_5^2+a_1^7a_3+a_1^4a_3^2+a_1^2a_3a_5+a_3a_7\not=0$,

\end{enumerate}

\noindent
we have  $\delta_{{\mathbb F}_{2^n}}(f)\geq 6$.
\end{theorem}

\begin{proof}
Since the differential uniformity of a polynomial is unchanged if it is multiplied by a non-zero scalar element, one can suppose that $f$ is monic i.e. $a_0=1$.
Conditions (i) and (iii) together with 
Proposition \ref{Morse} imply that $L_{a_1}f$ is a Morse polynomial of degree $d=3$.
But
the analogue of the Hilbert theorem given by Serre in Theorem 4.4.5 of \cite{Serre} (and detailled in even characteristic in the Appendix of Geyer in \cite{JardenRazon}) asserts that  the 
geometric monodromy group of a Morse polynomial  of degree $d$
is the symmetric group
$\Sgoth_d$. But since it is contained in its arithmetic monodromy group  which is also a subgroup of $\Sgoth_d$, they coincide.
Hence we deduce that the extension $F/{\mathbb F}_{2^n}(t)$ is geometric.

Moreover, Proposition 4.6 of \cite{YvesFabienFelipe} gives us that the extension $\Omega/F$ will be geometric if there exists 
$x\in{\mathbb F}_{2^n}$ such that $x^2+\alpha x=b_1/b_0$, where the $b_i$'s are given by
$L_{\alpha}f(x)=\sum_{k=0}^db_{d-k}x^k$. In our case, the equation reduces to
$x^2+a_1x=(a_1^2a_4+a_1a_5)/a_1a_3$.
Hilbert'90 theorem implies that the equation $x^2+ a_1 x = \frac{a_1a_4+a_5}{a_3}$ has a solution in $\mathbb{F}_{2^n}$ if and only if ${\Tr}_{{\mathbb F}_{2^n}/{\mathbb F}_2}  \left( \frac{a_1a_4+a_5}{a_1^2a_3} \right)=0$, which is exactly condition (ii) of the theorem.

Thus Proposition 4.6 of \cite{YvesFabienFelipe}  implies that the extension $\Omega/F$ is geometric.
Then we can apply the effective version of the Chebotarev density theorem given by Pollack in \cite{Pollack}
 to get the following lower bound (depending on $n$, the degree $d_{\Omega}$ of the extension $\Omega/{\mathbb F}_{2^n}(t)$ and the genus $g_{\Omega}$ of the function field $\Omega$)
 for the number $V$ of places of degree one in ${\mathbb F}_{2^n}(t)$ which  totally split in $\Omega$ (see for more details the proof of Theorem 4.1 of \cite{AHI-J.Algebra}):
 $$V\geq \frac{2^n}{d_{\Omega}}-\frac{2}{d_{\Omega}}(g_{\Omega}2^{n/2}+g_{\Omega}+d_{\Omega}).$$

 If $n$ is sufficiently large, this number is at least one.
 To be explicit, we have seen above that $G=\overline G={\frak S}_3$ and moreover,
 by 
 Proposition 4.6 of \cite{YvesFabienFelipe}, we have that $\Gamma=\overline{\Gamma}=({\mathbb Z}/2{\mathbb Z})^2$, so 
 $d_{\Omega}=3!\times 2^{2}=24$.
 Hence $V\geq 1$ as soon as $2^n-2g_{\Omega} 2^{n/2} -2g_{\Omega}-72 >0$.
 
 Now by Lemma 14 of \cite{Pollack} we have 
 $g_{\Omega}\leq \frac{1}{2}(\deg D_{\alpha}f-3)d_{\Omega}+1=37$.
 Hence if $n\geq 13$ we have $V\geq 1$ and this
  gives the existence of $\beta\in {\mathbb F}_{2^n}$ such that the polynomial 
$D_{\alpha}f(x)+\beta$ 
splits in $ \mathbb{F}_{2^n}[x]$ with no repeated factors. The differential uniformity of $f$ is thus greater than or equal to the degree of $D_{\alpha}f$, which is 6 in our present case.
\end{proof}

It implies for example that the polynomial
$f(x)=x^{10}+x^9+x^7+x^3$
has a differential uniformity over ${\mathbb F}_{2^n}$ greater than or equal to 6 pour $n\geq 13$.

\begin{corollary}
All polynomials
$$f(x)=x^{10}+a_1x^9+a_2x^8+a_3x^7+a_6x^4+a_7x^3+a_8x^2+a_9x+a_{10}$$
with $a_1, a_3$ in ${\mathbb F}_{2^n}^{\ast}$ and
$a_2, a_6, a_7, a_8, a_9,a_{10}$ in ${\mathbb F}_{2^n}$  and such that 
$a_7\not=a_1^7+a_1^4a_3$
have a differiential uniformity over ${\mathbb F}_{2^n}$ greater than or equal to 6 pour $n$ sufficiently large.
\end{corollary}













\section{The case  with $a_1=0$ and $a_3=0$}\label{Nul}

Making the choice $\alpha=a_1/a_0$ in the previous section gave a polynomial $D_{\alpha}f$ of degree 6, so the number of solutions
of any equation $D_{\alpha}f(x)=\beta$ could be at most 6.
If we choose $\alpha\not=a_1/a_0$ then the polynomial $D_{\alpha}f$ will be of degree 8 and  the equation $D_{\alpha}f(x)=\beta$ can have 8 solutions.
Let us study what happens in a particular case of this situation.

Suppose without loss of generality that $a_0=1$ and let $\alpha\in{\mathbb F}_{2^n}^{\ast}$ be such that $\alpha +a_1\not=0$ i.e. $\alpha\not=a_1$.
Then, by Formulas (\ref{Dalphaf}) and (\ref{Lalphaf}), we deduce that $D_{\alpha}f$ has degree 8 and $L_{\alpha}f$ has degree $d=4$.
The following proposition gives conditions for the algebraic and geometric monodromy groups of $\frac{1}{\alpha^2}L_{\alpha}f(x)$
to be the Klein group ${\mathbb Z}/2{\mathbb Z}\times{\mathbb Z}/2{\mathbb Z}$.

\begin{proposition}\label{Klein}
Let $f=\sum_{i=0}^{10}a_{10-i}x^i\in{\mathbb F}_{2^n}[x]$ be a polynomial
of degree $10$ with $a_0=1$, $a_1=a_3=0$.
Let $\alpha\in{\mathbb F}_{2^n}^{\ast}$ and set
$b:=\frac{\alpha^5+\alpha a_4+ a_5}{\alpha}$ and $c:= \frac{\alpha^2 a_5+ a_7}{\alpha}$.
Suppose that $c\not=0$ and that the polynomial $R_3(x):=x^3+bx^2+c^2$ 
factors over ${\mathbb F}_{2^n}$ as the product of three linear factors
(which means that ${\Tr}_{{\mathbb F}_{2^n}/{\mathbb F}_2}  \left( \frac{b^3}{c^2} \right)={\Tr}_{{\mathbb F}_{2^n}/{\mathbb F}_2}(1)$ and the roots of the polynomial $Q(T):=T^2+c^2T+b^6$ 
 are cubes in ${\mathbb F}_{2^n}$
(respectively in  ${\mathbb F}_{2^{2n}}$)
 if $n$ is even (respectively if $n$ is odd).

 Then
the quartic polynomial
$\frac{1}{\alpha^2}L_{\alpha}f(x)$
has algebraic and geometric monodromy groups isomorphic to the Klein group.
\end{proposition}

\begin{proof}
If we suppose that $a_0=1$ and $a_1=a_3=0$,
 then we get by Formula (\ref{Lalphaf}), for any $\alpha\in{\mathbb F}_{2^n}^{\ast}$:
 
 \begin{align*}
 L_{\alpha}f(x)&=\alpha^2 x^4 +(\alpha^6+\alpha^2a_4+\alpha a_5)x^2  +
 (\alpha^3 a_5+\alpha a_7)x \\
 & +
 \alpha^{10}+\alpha^8a_2+\alpha^6a_4+\alpha^5 a_5+\alpha^4a_6+\alpha^3 a_7+\alpha^2a_8+\alpha a_9
 \end{align*}

We set $g:=\frac{1}{\alpha^2}L_{\alpha}f$ and we
consider the irreducible polynomial
$$g(x)-t=\frac{1}{\alpha^2}L_{\alpha}f(x)-t \in{\mathbb F}_{2^n}(t)[x]$$
(recall that any polynomial $P(x)\in {\mathbb F}_{2^n}[x]$ gives rise to an irreducible  polynomial $P(x)-t$ in the ring
${\mathbb F}_{2^n}(t)[x]$).
We have:
 \begin{align*}
g(x)-t & =x^4+\frac{\alpha^5+\alpha a_4+ a_5}{\alpha} x^2
+ \frac{\alpha^2 a_5+ a_7}{\alpha} x \\
& + \frac{\alpha^{9}+\alpha^7a_2+\alpha^5a_4+\alpha^4 a_5+\alpha^3a_6+\alpha^2 a_7+\alpha a_8+ a_9}{\alpha}
+t.
 \end{align*}
So we have 
$$g(x)-t=x^4+b x^2 + c x + d$$
 with $b:=\frac{\alpha^5+\alpha a_4+ a_5}{\alpha}$, $c:= \frac{\alpha^2 a_5+ a_7}{\alpha}$
and $d:=\frac{\alpha^{9}+\alpha^7a_2+\alpha^5a_4+\alpha^4 a_5+\alpha^3a_6+\alpha^2 a_7+\alpha a_8+ a_9}{\alpha}
+t$.

The monic quartic polynomial $g(x)-t$ in ${\mathbb F}_{2^n}(t)[x]$ with no cubic term is separable if and only if $c\not=0$ 
(see  the illustration of Theorem 3.4.  of \cite{Conrad})
and  its quadratic resolvent $R_2(x)$ and its cubic resolvent $R_3(x)$ are given by (see  equations (3.4) and (3.5) of \cite{Conrad}):
$$R_2(x)=x^2+c^2x+(b^3+c^2)c^2$$
and
$$R_3(x)=x^3+bx^2+c^2.$$

It is well-known that $R_2(X)$ is reducible if and only if 
${\Tr}_{{\mathbb F}_{2^n}/{\mathbb F}_2}  \left( \frac{(b^3+c^2)c^2}{c^4} \right)=0$
i.e.
${\Tr}_{{\mathbb F}_{2^n}/{\mathbb F}_2}  \left( \frac{b^3}{c^2} \right)={\Tr}_{{\mathbb F}_{2^n}/{\mathbb F}_2}(1).$

Let us consider now the reducibility of the polynomial $R_3(x)=x^3+bx^2+c^2$.
The substitution $z=x+b$ eliminates the quadratic term: it gives the equation
$z^3+b^2z+c^2=0$.

Theorem 1 of \cite{Williams} gives that the polynomial $z^3+b^2z+c^2$ (with $c\not= 0$) is reducible if and only if

(i) ${\Tr}_{{\mathbb F}_{2^n}/{\mathbb F}_2}  \left( \frac{b^6}{c^4} \right)\not= {\Tr}_{{\mathbb F}_{2^n}/{\mathbb F}_2}(1)$ (in this case the polynomial has a unique root in ${\mathbb F}_{2^n}$),

 or

(ii) ${\Tr}_{{\mathbb F}_{2^n}/{\mathbb F}_2}  \left( \frac{b^6}{c^4} \right)={\Tr}_{{\mathbb F}_{2^n}/{\mathbb F}_2}(1)$ and the roots of the polynomial $Q(T):=T^2+c^2T+b^6$ are cubes in ${\mathbb F}_{2^n}$ if $n$ is even, or in ${\mathbb F}_{2^{2n}}$ if $n$ is odd (in this case the polynomial $z^3+b^2z+c^2$
factors over ${\mathbb F}_{2^n}$ as the product of three linear factors).

So if 
$\alpha\in{\mathbb F}_{2^n}^{\ast}$ is such that 
${\Tr}_{{\mathbb F}_{2^n}/{\mathbb F}_2}  \left( \frac{b^6}{c^4} \right)={\Tr}_{{\mathbb F}_{2^n}/{\mathbb F}_2}(1)$,
i.e. 
${\Tr}_{{\mathbb F}_{2^n}/{\mathbb F}_2}  \left( \frac{b^3}{c^2} \right)={\Tr}_{{\mathbb F}_{2^n}/{\mathbb F}_2}(1)$,
and also such that the roots of the polynomial $Q(T)$ are cubes in ${\mathbb F}_{2^n}$ or in ${\mathbb F}_{2^{2n}}$ (according as
$n$ is even or  odd), then 
 the polynomials $R_2(x)$ and $R_3(x)$ are reducibles.

Finally,  with the hypothesis of the proposition, $g(x)-t$ is a separable irreducible quartic polynomial of ${\mathbb F}_{2^n}(t)[x]$ such that its quadratic and cubic resolvents are reducibles. By Theorem 3.4. of \cite{Conrad}, we obtain that 
the Galois group $G_g$ of the polynomial $g(x)-t=\frac{1}{\alpha^2}L_{\alpha}f(x)-t$, which is
the arithmetic monodromy group 
of the polynomial $g(x)=\frac{1}{\alpha^2}L_{\alpha}f(x)$,
 is isomorphic to the Klein group ${\mathbb Z}/2{\mathbb Z}\times{\mathbb Z}/2{\mathbb Z}$.

Since the polynomial $g(x)-t$ is irreducible over ${\mathbb F}_{2^n}(t)$, the arithmetic and the geometric monodromy groups of $\frac{1}{\alpha^2}L_{\alpha}f(x)$, seen as permutation groups, are transitive subgroups of the symmetric group $\Sgoth_4$. 
It is well-known (see \cite{Conrad} for example) that the only transitive subgroups of $\Sgoth_4$ are $\Sgoth_4$ himself, the alternate group $\Agoth_4$, 
three conjugate subgroups isomorphic to the dihedral group $D_4$ of order 8, three conjugate subgroups isomorphic to the cyclic group ${\mathbb Z}/4{\mathbb Z}$ and one subgroup isomorphic to the Klein group ${\mathbb Z}/2{\mathbb Z}\times{\mathbb Z}/2{\mathbb Z}$.

Since the geometric monodromy group $\overline{G}_g$ of $g(x)=\frac{1}{\alpha^2}L_{\alpha}f(x)$ is a normal subgroup of $G_g$ and a transitive subgroup of $\Sgoth_4$, we obtain that $\overline{G}_g$ is also the Klein group ${\mathbb Z}/2{\mathbb Z}\times{\mathbb Z}/2{\mathbb Z}$.
\end{proof}



\begin{remark}\label{separable}
The condition $c\not=0$ in the previous theorem is equivalent to saying that the polynomial
$g(x)-t:=\frac{1}{\alpha^2}L_{\alpha}f(x)-t \in{\mathbb F}_{2^n}(t)[x]$
is separable (see  the illustration of Theorem 3.4.  of \cite{Conrad}).
\end{remark}

\begin{remark}\label{cubes}
The condition in the previous theorem saying that
the polynomial $R_3(x):=x^3+bx^2+c^2$ 
factors over ${\mathbb F}_{2^n}$ as the product of three linear polynomials
 is equivalent to saying that (see Theorem 1 of \cite{Williams}):
${\Tr}_{{\mathbb F}_{2^n}/{\mathbb F}_2}  \left( \frac{b^3}{c^2} \right)={\Tr}_{{\mathbb F}_{2^n}/{\mathbb F}_2}(1)$ and the roots of the equation $T^2+c^2T+b^6$ are cubes in ${\mathbb F}_{2^n}$
(respectively in  ${\mathbb F}_{2^{2n}}$)
 if $n$ is even (respectively if $n$ is odd).
\end{remark}


\begin{example}\label{Klein-bis}
Let $f=\sum_{i=0}^{10}a_{10-i}x^i\in{\mathbb F}_{2^n}[x]$ be a polynomial
of degree $10$ with $a_0=1$, $a_1=a_3=a_4=a_5=0$ and $a_7=1$, i.e. the polynomial $f$
has the form
$$f(x)=x^{10}+a_2x^8+a_6x^4+x^3+a_8x^2+a_9x+a_{10}$$
with $a_2, a_6, a_8, a_9,a_{10}$ in ${\mathbb F}_{2^n}$.
Let us show that if $n\equiv 0\pmod 4$  then there exists $\alpha\in{\mathbb F}_{2^n}^{\ast}$ 
such that
the polynomial
$\frac{1}{\alpha^2}L_{\alpha}f(x)$
has algebraic and geometric monodromy groups isomorphic to the Klein group.

Indeed, let $\alpha\in{\mathbb F}_{2^n}^{\ast}$ and consider, as in the proof of  Proposition \ref{Klein}, the irreducible polynomial
$$g(x)-t:=\frac{1}{\alpha^2}L_{\alpha}f(x)-t \in{\mathbb F}_{2^n}(t)[x].$$ 
So we have 
$$g(x)-t=x^4+b x^2 + c x + d$$
 with $b:=\alpha^4$, $c:= \frac{1}{\alpha}$ and 
 $d:=\frac{\alpha^{9}+\alpha^7a_2+\alpha^3a_6+\alpha^2 +\alpha a_8+ a_9}{\alpha}
+t$.

 Since  $c\not= 0$ then, by Remark \ref{separable},  the polynomial $g(x)-t$ is separable.
Moreover, the condition ${\Tr}_{{\mathbb F}_{2^n}/{\mathbb F}_2}  \left( \frac{b^3}{c^2} \right)={\Tr}_{{\mathbb F}_{2^n}/{\mathbb F}_2}(1)$ 
in Proposition \ref{Klein} remains to
$${\Tr}_{{\mathbb F}_{2^n}/{\mathbb F}_2}  (\alpha^{7})={\Tr}_{{\mathbb F}_{2^n}/{\mathbb F}_2}(1)=n\pmod 2.$$

Now the equation $T^2+c^2T+b^6=0$ becomes 
$$T^2+\frac{1}{\alpha^2}T+\alpha^{24}=0.$$
We are looking for $\alpha$ in ${\mathbb F}_{2^n}^{\ast}$ such that the solutions of this equation are cubes in ${\mathbb F}_{2^n}$. 
Note that these roots belong to ${\mathbb F}_{2^n}$ if and only if  ${\Tr}_{{\mathbb F}_{2^n}/{\mathbb F}_2}(\alpha^{28})=0$, i.e. ${\Tr}_{{\mathbb F}_{2^n}/{\mathbb F}_2}(\alpha^7)=0$.

\bigskip

But one can show that there exists $\alpha\in {\mathbb F}_{2^4}^{\ast}$ such that the polynomial $T^2+\frac{1}{\alpha^2}T+\alpha^{24}$ has roots  which are cubes in ${\mathbb F}_{16}^{\ast}$ and with
${\Tr}_{{\mathbb F}_{2^4}/{\mathbb F}_2}(\alpha^7)=0$.
Indeed, take  ${\mathbb F}_{16}={\mathbb F}_2[X]/(X^4+X^3+1)={\mathbb F}_2(\theta)$ and choose 
$\alpha=\theta^{10}$. Then 
$$Q(T)=T^2+\theta^{10} T+1=T^2+\frac{1}{(\theta^{10})^2} T + (\theta^{10})^{24}=(T+(\theta^2)^3)(T+(\theta^{3})^3)$$
with
$${\Tr}_{{\mathbb F}_{2^4}/{\mathbb F}_2}(\alpha^7)={\Tr}_{{\mathbb F}_{2^4}/{\mathbb F}_2}(\theta^{70})=
{\Tr}_{{\mathbb F}_{2^4}/{\mathbb F}_2}(\theta^{10})=
{\Tr}_{{\mathbb F}_{2^4}/{\mathbb F}_2}(\theta^{5})=
{\Tr}_{{\mathbb F}_{2^4}/{\mathbb F}_2}(\alpha^2)
=0.
$$

In conclusion, if $f=\sum_{i=0}^{10}a_{10-i}x^i\in{\mathbb F}_{2^n}[x]$ is a polynomial
of degree $10$ with $a_0=a_7=1$ and $a_1=a_3=a_4=a_5=0$, 
and if $n\equiv 0\pmod 4$  there exists $\alpha\in{\mathbb F}_{2^n}^{\ast}$ (since in this case ${\mathbb F}_{16}$ is included in ${\mathbb F}_{2^n}$)
such that
$c\not=0$ and,
by Remark \ref{cubes},
such that  the polynomial $R_3(x):=x^3+bx^2+c^2$ 
has all its  roots in ${\mathbb F}_{2^n}$.
Hence by Proposition \ref{Klein}
the polynomial
$\frac{1}{\alpha^2}L_{\alpha}f(x)$
has algebraic and geometric monodromy groups isomorphic to the Klein group.
\end{example}

\bigskip
\bigskip

Recall that $F$ is the splitting field  of the
polynomial 
$L_{\alpha}f(x)-t$
 over the field ${\mathbb F}_{2^n}(t)$ and 
 $\Omega={\mathbb F}_{2^n}(x_0,\ldots, x_{d-1})$ is
 the compositum of the fields $F(x_i)$, where 
 $u_0,\ldots,u_{d-1}$ are the roots of $L_{\alpha}f(x)=t$
 and
 $x_i$
are the roots of
$x^2+ \alpha x=u_i$.

Now let us give a sufficient condition for the extension $\Omega/F$ to be geometric.

\begin{lemma}\label{Second_Floor}
Let $f=\sum_{i=0}^{10}a_{10-i}x^i\in{\mathbb F}_{2^n}[x]$ be a polynomial
of degree $10$ with $a_0=1$, $a_1=a_3=0$.
Let $\alpha\in{\mathbb F}_{2^n}^{\ast}$ and set
$b:=\frac{\alpha^5+\alpha a_4+ a_5}{\alpha}$ and $c:= \frac{\alpha^2 a_5+ a_7}{\alpha}$.
Suppose that $c\not=0$ and that the polynomial $R_3(x):=x^3+bx^2+c^2$ 
factors over ${\mathbb F}_{2^n}$ as the product of three linear factors.

Then the extension $\Omega/F$ is geometric 
  as soon as  the equation $x^2+ \alpha x = \frac{\alpha^5+\alpha a_4 + a_5}{\alpha}$
has a solution in $\mathbb{F}_{2^n}$.
\end{lemma}

\begin{proof}
We begin proving that 
if $u$ is a root of $L_{\alpha}f(x)-t$ in $F$, then,
for each place $\wp$ of $F$ above the place ${\infty}$ at infinity
of ${\mathbb F}_{2^n}(t)$, we have
 that $u$ has a simple pole at $\wp$.
 
 Indeed, the  field ${\mathbb F}_{2^n}(t)(u)$ is just the rational function field ${\mathbb F}_{2^n}(u)$.
 The place at infinity $P_{\infty}$ of ${\mathbb F}_{2^n}(u)$ is the pole of $u$ and it is the place above the place at infinity ${\infty}$ of ${\mathbb F}_{2^n}(t)$ (which corresponds to the  pole of $t$).
 Thus the valuation of $u$ at  $P_{\infty}$ is given by $v_{P_{\infty}}(u)=-1$ and therefore $v_{P_{\infty}}(L_{\alpha}f(u))=-\deg(L_{\alpha}f(x))$.
 Since the ramification index $e(P_{\infty} \vert {\infty})$ of $P_{\infty}$ over ${\infty}$ 
 verify:
 $$v_{P_{\infty}}(L_{\alpha}f(u))=v_{P_{\infty}}(t)=e(P_{\infty} \vert {\infty}) v_{\infty}(t)=-e(P_{\infty} \vert {\infty})$$
  thus we obtain:
 $$e(P_{\infty} \vert {\infty})=\deg(L_{\alpha}f(x))=4.$$

 But 
 the hypotheses on $c$ and $R_3(x)$ imply
 by Proposition \ref{Klein} that the Galois extension $F/{\mathbb F}_{2^n}(t)$ has Galois group the Klein group of order 4
 ( the place at infinity of ${\mathbb F}_{2^n}(t)$ is then totally ramified in ${\mathbb F}_{2^n}(u)$).
 We conclude that
 $F={\mathbb F}_{2^n}(u)$ and then $u$ has a simple pole at $\wp=P_{\infty}$.
 
 Now we show that if 
 $J \subset \{ 0, 1, 2, 3 \}$ is neither empty nor the whole set then
$ \sum_{j \in J} u_j $  has a pole at the place at infinity $P_{\infty}$ of $F$.
Since the coefficient of $x^3$ in the polynomial $L_{\alpha}f$ is zero (see Formula (\ref{Lalphaf})), we have that
 $u_0+u_1+u_2+u_3=0$.
 We are then reduced to show that $u_0+u_1$, $u_0+u_2$ and $u_0+u_3$ have a pole at $P_{\infty}$.
 But we are in the situation where the Galois extension $F/{\mathbb F}_{2^n}(t)$ has a Galois group 
 isomorphic to 
 ${\mathbb Z}/2{\mathbb Z}\times{\mathbb Z}/2{\mathbb Z}$, so the following diagram summarize the situation (where $k:={\mathbb F}_{2^n}$ and all the extensions have degree 2).

 
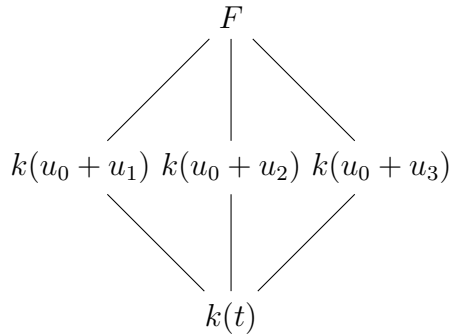
\begin{center}
\begin{tikzpicture}[node distance=2cm]
 \node (k)                  {$k(t)$};
 \node (k1plus2) [above of=k]  {$k(u_0+u_2)$};
 \node (k12)  [above of=k1plus2]   {$F$};
 \node (k2)  [right of=k1plus2]  {$k(u_0+u_3)$};
 \node (k1)  [left of=k1plus2]   {$k(u_0+u_1)$};
 \draw (k)   -- (k2);
 \draw (k)   -- (k1plus2);
 \draw (k)   -- (k1);
 \draw (k2)  -- (k12);
 \draw (k1)  -- (k12);
 \draw (k1plus2)  -- (k12);
\end{tikzpicture}
\end{center}

If we denote by $\infty_i$ the place at infinity of ${\mathbb F}_{2^n}(u_0+u_i)$, for $i=1,2,3$, we have that the ramification index 
$e(P_{\infty}\vert \infty_i)=e(\infty_i\vert \infty)=2$ for all $i$ since $\infty$ is totally ramified in the extension $F/{\mathbb F}_{2^n}(t)$.
 
So 
we have 
$$v_{P_{\infty}}(u_{0}+u_{i})=e(P_{\infty}\vert \infty_i) v_{\infty_i}(u_0+u_i) = 2\times (-1)=-2\leq -1$$
which proves that $P_{\infty}$ is a pole of $u_0+u_i$.

Then, the proof of Proposition 4.6 of \cite{YvesFabienFelipe} remains true with polynomials of degree 10 with geometric and arithmetic monodromy groups the Klein group:
if there exists $x \in {\mathbb F}_{2^n}$ such that 
$x^2+ \alpha x = b_1/b_0$
where the $b_i's$ are defined by $\frac{1}{\alpha^2}L_{\alpha}f(x)=\sum_{i=0}^4b_{4-i}x^i$
 then 
$\Gal \left(  F(x_0, x_1, x_2,x_{3}) /  F \right)$ 
and
$\Gal \left(  F{\mathbb F}_{2^n}^{\Omega}(x_0, x_1, x_2,x_{3}) /  F{\mathbb F}_{2^n}^{\Omega} \right)$ 
 are isomorphic to  $\left( {\mathbb Z} / 2{\mathbb Z} \right)^{3}$,
 where ${\mathbb F}_{2^n}^{\Omega}$ denotes the algebraic closure of ${\mathbb F}_{2^n}$ in $\Omega$
 and $F{\mathbb F}_{2^n}^{\Omega}$ denotes the compositum of the fields $F$ and ${\mathbb F}_{2^n}^{\Omega}$.
  The coefficients $b_i$'s come from Equation (\ref{Lalphaf}): $b_1/b_0=\frac{\alpha^5+\alpha a_4 + a_5}{\alpha}$, and   
 the existence of a solution in ${\mathbb F}_{2^n}$ of the equation 
$x^2+ \alpha x = b_1/b_0$
is exactly the last condition of the Lemma.
Thus we conclude that the extension $\Omega/F$ 
 is geometric.

\end{proof}

\begin{theorem}\label{cas_a_1=a_3=0}
Let $f=\sum_{i=0}^{10}a_{10-i}x^i\in{\mathbb F}_{2^n}[x]$ be a polynomial
of degree $10$ with $a_1=a_3=0$.

Suppose that there exists $\alpha\in{\mathbb F}_{2^n}^{\ast}$ such that:
\begin{enumerate}[label=(\roman*)]
\item
$c:= \frac{\alpha^2 a_5+ a_7}{\alpha}\not=0$
 and the polynomial $R_3(x):=x^3+bx^2+c^2$ 
has all its  roots in ${\mathbb F}_{2^n}$ where
$b:=\frac{\alpha^5+\alpha a_4+ a_5}{\alpha}$,

\item
 and ${\Tr}_{{\mathbb F}_{2^n}/{\mathbb F}_2}  \left( \frac{\alpha^5+\alpha a_4 + a_5}{\alpha^3}\right)=0$.
\end{enumerate}
Then $\delta_{{\mathbb F}_{2^n}}(f)=8$ if $n$ is sufficiently large (namely if $n\geq 15$).
\end{theorem}

\begin{proof}
Let $f$ be a polynomial as in the theorem.
Looking at its differential uniformity, one can suppose that $f$ is monic.
Condition (i) implies by Proposition \ref{Klein} that the polynomial
$\frac{1}{\alpha^2}L_{\alpha}f(x)$
has algebraic and geometric monodromy groups isomorphic to the Klein group.
Hence the splitting field $F$ of the polynomial $g(x):=\frac{1}{\alpha^2}L_{\alpha}f(x)-t$ is a geometric extension of ${\mathbb F}_{2^n}(t)$.
  

 Moreover, by Lemma \ref{Second_Floor}, the extension $\Omega/F$ is geometric 
  as soon as  the equation 
  $x^2+ \alpha x = \frac{\alpha^5+\alpha a_4 + a_5}{\alpha}$
has a solution in $\mathbb{F}_{2^n}$.
By 
   the Hilbert'90 Theorem, this is equivalent to
   ${\Tr}_{{\mathbb F}_{2^n}/{\mathbb F}_2}  \left( \frac{\alpha^5+\alpha a_4 + a_5}{\alpha^3}\right)=0$,
   which is precisely Condition (ii).

Then we use the Chebotarev theorem, as in the proof of Theorem \ref{Main}, to obtain,
if $n$ is sufficiently large (namely here if $n\geq 15$), the existence of $\beta\in {\mathbb F}_{2^n}$ such that the polynomial 
$D_{\alpha}f(x)+\beta\alpha^2$ 
splits in $ \mathbb{F}_{2^n}[x]$ with no repeated factors.

Thus the differential uniformity of $f$ is 
equal to the degree of $D_{\alpha}f$ that is 8.
\end{proof}

\begin{example}
Let us come back to Example \ref{Klein-bis}, and since the differential uniformity is unchanged if we add an additive polynomial,
let us just consider the polynomial $f(x)=x^{10}+x^3\in{\mathbb F}_{2^n}[x]$.
We have seen that,
if $n\equiv 0\pmod 4$, then
 there exists
$\alpha\in{\mathbb F}_{16}^{\ast}\subset {\mathbb F}_{2^n}^{\ast}$ 
 such that the polynomial $T^2+\frac{1}{\alpha^2}T+\alpha^{24}$ has roots  which are cubes in ${\mathbb F}_{16}^{\ast}$ and 
with ${\Tr}_{{\mathbb F}_{16}/{\mathbb F}_2}(\alpha^7)={\Tr}_{{\mathbb F}_{16}/{\mathbb F}_2}(\alpha^2)=0$.
Hence  there exists
$\alpha\in {\mathbb F}_{2^n}^{\ast}$ 
such that the
 polynomial
$\frac{1}{\alpha^2}L_{\alpha}f(x)$
has algebraic and geometric monodromy groups isomorphic to the Klein group.
Moreover 
 the equation $x^2+ \alpha x = \frac{b_1}{b_0}$
has a solution in $\mathbb{F}_{2^n}$ 
since 
${\Tr}_{{\mathbb F}_{2^n}/{\mathbb F}_2}  \left( \frac{\alpha^5+\alpha a_4 + a_5}{\alpha^3} \right)=
{\Tr}_{{\mathbb F}_{2^n}/{\mathbb F}_2}(\alpha^2)
=0$.
Finally we conclude by Theorem \ref{cas_a_1=a_3=0}
that if $n$ is sufficiently large and $n\equiv 0\pmod 4$  then $\delta_{{\mathbb F}_{2^n}}(f)=8$.

\end{example}




\bigskip
\noindent
{\bf Acknowledgements:}  The author would like to thank Fabien Herbaut for many discussions on this topic.

\bigskip
\bibliographystyle{plain}

\end{document}